\documentclass[a4paper,12pt]{article}

\usepackage{amsmath}
\usepackage{amsfonts} 
\usepackage{amssymb,amsthm}
\usepackage{a4wide}
\usepackage[english]{babel}   

\newtheorem{de}{Definition}[section]

\newtheorem{prop}[de]{Proposition}
\newtheorem{cor}[de]{Corollary}
\newtheorem{lem}[de]{Lemma}

\begin{document}

\title{Complex supermanifolds with many unipotent automorphisms}
\author{
{\bf Matthias Kalus}\\
{\small  Fakult\"at f\"ur Mathematik}\\{\small Ruhr-Universit\"at Bochum}\\ {\small D-44780 Bochum, Germany}\\
}
\date{}
\maketitle
\centerline{ {\bf Keywords: } complex supermanifold; nilpotent vector field; unipotent automorphism  }
\centerline{  {\bf MSC2010:} 58A50, 58H15, 54H15 }
\renewcommand{\thefootnote}{\arabic{footnote}}

\begin{abstract}\noindent
An automorphism on a complex supermanifold $\mathcal M$ is called unipotent if it reduces to the identity on the associated graded supermanifold $gr(\mathcal M)$.  These automorphisms are close to be complementary to those responsible for homogeneity of a supermanifold. In analogy, their study yields results on the classification of supermanifolds. Unipotent automorphisms are induced by even global degree increasing vector fields $X\in \mathcal V_{\mathcal M,\bar 0}^{(2)}$. Plenitude of unipotent automorphisms is understood as follows: the presheaf of common kernels of the operators  $[X,\cdot]$ for $X\in \mathcal V_{\mathcal M,\bar 0}^{(2)}$, on  superderivations vanishes up to errors of a fixed degree $t$ and higher. The isomorphy class of such strictly $t$-nildominated supermanifolds is determined up to errors of degree $t$ and higher by $\mathcal V_{\mathcal M,\bar 0}^{(2)}$ and $gr(\mathcal M)$. An example shows that a strictly $t$-nildominated supermanifold can be non-split, deformed  already in degrees lower 
than $t$. 
\end{abstract}

\bigskip \noindent
A complex supermanifold $\mathcal M$ is called homogeneous if it admits a transitive action of a Lie supergroup $\mathcal G$. Here transitivity is defined by two conditions: first  the underlying action $G \times M\to M$ is transitive, in particular the even vector fields cover all (even) directions of the underlying manifold. Secondly the induced odd vector fields on $\mathcal M$ cover all odd directions, i.e. all fiber directions of the associated vector bundle $E\to M$. Complex homogeneous supermanifolds were studied in \cite{On96} and considering special classes of underlying manifolds e.g. in \cite{On07}. A classification statement is given in \cite{Vi11}: a complex homogeneous supermanifold can be biholomorphically identified with a quotient $\mathcal G/\mathcal H$ of complex Lie supergroups. Further a condition for the existence of splittings can be found in \cite{Vi15}.

\bigskip \noindent
In the present paper we consider the global even vector fields on a complex supermanifold $\mathcal M$ that reduce on the associated graded supermanifold $gr(\mathcal M)=(M,\mathcal O_{\Lambda E})$ to zero. These fields are in some sense complementary to those studied on homogeneous supermanifolds. Their Lie algebra $\mathcal V_{\mathcal M,\bar 0}^{(2)}$ is bijectively identified via the exponential map with the group $H^0(M,G_E)$ of unipotent automorphisms of $\mathcal M$, i.e. automorphisms reducing to the identity on $gr(\mathcal M)$. $H^0(M,G_E)$ is a normal divisor in the group of automorphisms of $\mathcal M$, and in the case of a compact underlying manifold, it is a complex Lie subgroup of the automorphism Lie supergroup of $\mathcal M$ (see \cite{BK}). The study of supermanifolds with many unipotent automorphisms yields the counterpart of the classification results on supermanifolds with no unipotent automorphisms in \cite{Kal}. In detail, we give a condition under which the equivalence class of a supermanifold is encoded in the nilpotent Lie group $H^0(M,G_E)$ and the complex vector bundle $E\to M$.

\bigskip\noindent
We give suitable notions for the plenitude of vector fields in $\mathcal V_{\mathcal M,\bar 0}^{(2)}$ and deduce results on the classification and the splitting problem of complex supermanifolds in the non-homogeneous case. Elements of $\mathcal V_{\mathcal M,\bar 0}^{(2)}$ can hardly be considered as tangent vectors at points of $M$. A first promising replacement for transitivity is the following definition of plenitude: denote the sheaf of even $\mathbb Z$-degree increasing derivations on the sheaf of superfunctions $\mathcal O_\mathcal M$ by $Der_{\bar 0}^{(2)}(\mathcal O_\mathcal M)$. A complex supermanifold $\mathcal M$ is called $2s$-nildominated  if the presheaf of common kernels of the $[X,\cdot]\subset Der_{\bar 0}^{(2)}(Der_{\bar 0}^{(2)}(\mathcal O_\mathcal M))$ for $X \in \mathcal V_{\mathcal M,\bar 0}^{(2)}$, is contained in degree $2s$ and higher in the $\mathbb Z$-filtration. We will prove the following result in section \ref{sec2}: fixing $E\to M$, the \v{C}ech  equivalence class $\alpha\in 
H^1(M,G_E)$ (see \cite{Gr}) of a $2s$-nildominated complex supermanifold $\mathcal M$ is 
fixed up to errors of degree $2s$ and higher by the explicit fields $\mathcal V_{\mathcal M,\bar 0}^{(2)}$ regarded as a subset of $C^0(M,Der_{\bar 0}^{(2)}(\mathcal O_{\Lambda E}))$.  

\bigskip\noindent
In section \ref{sec3} we discuss a graded version of nildominance naturally arising from the definition of nildominance. We prove that graded $2s$-nildominated supermanifolds are already split up to errors of degree $2s$ and higher.  So one might suspect that also $2s$-nildominated supermanifolds are split up to errors of degree $2s$ and higher. A disappointing expectation. In section \ref{sec4} we give an example of a $4$-nildominated supermanifold that is non-split being non-trivially deformed already in degree $2$. The example is constructed on $\mathbb P^1(\mathbb C)$ with odd dimension $7$. Although it might be suspected that many examples exist, it is technically difficult to construct one. Reasons are explained in the respective section. 

\bigskip\noindent
Unfortunately the notion of $2s$-nildominance is not stable with respect to products of supermanifolds, a feature that holds for homogeneity of supermanifolds. Section \ref{sec5} presents the final and stricter notion of plenitude of unipotent automorphisms. Here the presheaf of common kernels of the  $[X,\cdot]\subset Der_{\bar 0}^{(2)}(Der(\mathcal O_\mathcal M))$  for  $X \in \mathcal V_{\mathcal M,\bar 0}^{(2)}$, is asked to lie in degree $t$ and higher. A complex supermanifold satisfying this will be called strictly $t$-nildominated. We prove that strict $t$-nildominance of the factors induces strict $t$-nildominance of the product of complex supermanifolds. Finally strict $t$-nildominance still allows non-split examples: the example of section \ref{sec4} is shown to be strictly $3$-nildominated.

\section{Notation}

First we fix notation. We consider supermanifolds in the sense of Berezin, Kostant, and Leites. For details see e.g. \cite{Kost}, \cite{Lei}, \cite{D-M}. Let $\mathcal M$ be a complex supermanifold with sheaf of superfunctions $\mathcal O_\mathcal M$ containing the subsheaf $\mathcal N_\mathcal M$ of nilpotent elements. Denote the induced $\mathbb Z$-filtration by upper indexes in brackets, so $\mathcal O_\mathcal M^{(0)}=\mathcal O_\mathcal M$, $\mathcal O_\mathcal M^{(1)}=\mathcal N_\mathcal M$ etc. and note that the sheaves of superderivations $Der(\mathcal O_\mathcal M)$ and of endomorphisms $End(\mathcal O_\mathcal M)$ inherit the $\mathbb Z$-filtration. Denote the underlying manifold of $\mathcal M$ by $M$ and the associated complex vector bundle defined via $\mathcal O_\mathcal M^{(1)}/\mathcal O_\mathcal M^{(2)}$ by $E\to M$. The $\mathbb Z$-graded supermanifold associated to $\mathcal M$ is 
$ gr(\mathcal M)=(M,\mathcal O_{\Lambda E})$ with $O_{\Lambda E}$ denoting the sheaf of sections in $\Lambda E$. The $\mathbb 
Z$-grading will be denoted by lower indexes. Let $\alpha\in H^1(M,G_E)$ be an associated non-split cohomology class of $\mathcal M$ (see \cite{Gr}). Here $G_E$ is the sheaf of those (even) automorphisms of the sheaf of $\mathbb Z/2\mathbb Z$-graded algebras $\mathcal O_{\Lambda E}$ that have the identity as degree zero part. Further $H^1(M,G_E)$ denotes the first Check cohomology induced by concatenation of morphisms. Now $\alpha$ is well-defined up to the action of $H^0(M,Aut(E))$, the global automorphisms of $E$. Following \cite{Rot}, it is $\alpha=[\exp(u)]$ represented by a cochain $u$ in $C^1(M,Der^{(2)}_{\bar 0}(\mathcal O_{\Lambda E}))$. The lower index with bar denotes the $\mathbb Z/2\mathbb Z$-grading. If $u$ can be chosen to lie in $C^1(M,Der^{(2s)}_{\bar 0}(\mathcal O_{\Lambda E}))$ then $\mathcal M$ will be called \textit{split up to errors of degree} $2s$ \textit{and higher}. 

\bigskip
Let $\{U_i\}$ be a Leray covering of coordinate charts of $M$ with respect to coherent sheaf cohomology such that the $U_{ij}:=U_i\cap U_j$ are connected. Then superfunctions $f\in \mathcal O_\mathcal M(U_i\cup U_j)$ with $f_l=f|_{U_l} \in \mathcal O_{\Lambda E}|_{U_l}$, $l=i,j$, transform with $f_i =\alpha_{ij}(f_j)$. So for derivations $X=Der(\mathcal O_\mathcal M)(U_i\cup U_j)$ with $X_l=X|_{U_l} \in Der(\mathcal O_{\Lambda E}|_{U_l})$, $l=i,j$, we have: 
\begin{align}\label{eq:001}
 X_i=\alpha_{ij} X_j \alpha_{ij}^{-1}
\end{align}
Let $\mathcal V_{\mathcal M}=H^0(M, Der(\mathcal O_\mathcal M))$  be the Lie superalgebras of global super vector fields on $\mathcal M$ and $\mathcal V_{\Lambda E}=H^0(M, Der(\mathcal O_{\Lambda E}))$. Then $\mathcal V_{\mathcal M,\bar 0}^{(2)}$ is the Lie algebra of the group of unipotent automorphisms of $\mathcal M$. By the above considerations we can regard $\mathcal V_{\mathcal M,\bar 0}^{(2)}$ as a subset of $C^0(M,Der_{\bar 0}^{(2)}(\mathcal O_{\Lambda E}))$. Now (\ref{eq:001}) for $X$ running through $\mathcal V_{\mathcal M,\bar 0}^{(2)}$ yields conditions on $u=\log(\alpha)$. If these conditions  determine $u$ up to terms in $C^1(M,Der^{(2s)}_{\bar 0}(\mathcal O_{\Lambda E}))$ then $\mathcal M$ will be called \textit{determined up to errors of degree $2s$ and higher by its unipotent automorphisms}.

\bigskip
For an $X \in \mathcal V_{\mathcal M,\bar 0}^{(2)}$ set $k(X)$ to be the maximal positive integer with $X \in \mathcal V_{\mathcal M,\bar 0}^{(k(X))}$.  We obtain a linear map  $X_{\bullet}: \mathcal O_{\mathcal M }^{(t)}/\mathcal O_{\mathcal M }^{(t+1)}\to \mathcal O_{\mathcal M }^{(t+k(X))} / \mathcal O_{\mathcal M }^{(t+k(X)+1)}$ for any $t\geq 0$, and linearly continue it to the sheaf $\mathcal O_{\Lambda E}=\oplus_{t\geq 0} \mathcal O_{\mathcal M }^{(t)}/\mathcal O_{\mathcal M }^{(t+1)}$. The resulting object is a derivation  $X_{\bullet} \in H^0(M, Der^{k(X)}_{\bar 0}(\mathcal O_{\Lambda E}))$. Note that $\mathcal V_{\mathcal M,\bar 0}^{(2)} \to H^0(M, Der^{(2)}_{\bar 0}(\mathcal O_{\Lambda E}))$, $X\mapsto X_\bullet$ is in general not linear.

\section{Nildominated supermanifolds}\label{sec2}

We give sense to the expression of plenitude of unipotent automorphisms. 

\begin{de}
 A complex supermanifold $\mathcal M$ is called $2s$-\textit{nildominated}, $s \in \mathbb N$, if for any open set $U\subset M$ we have:
\begin{align}\label{eq:003}
 \bigcap_{X\in \mathcal V_{\mathcal M,\bar 0}^{(2)}} 
 Ker\Big([X,\cdot]:Der^{(2)}_{\bar 0}(\mathcal O_\mathcal M) (U)
 \to Der^{(2)}_{\bar 0}(\mathcal O_\mathcal M) (U)\Big)\subset Der^{(2s)}_{\bar 0}(\mathcal O_\mathcal M) (U)
\end{align}
\end{de}

In particular $2s$-nildominance forces the non-trivial center of $\mathcal V_{\mathcal M,\bar 0}^{(2)}$  to lie in $\mathcal V_{\mathcal M,\bar 0}^{(2s)}$.  If the odd dimension of $\mathcal M$ is $2s$ or $2s+1$ then $\mathcal M$ can be at most $2s$-nildominated. In this case $\mathcal M$ is simply called \textit{nildominated}. We can follow:

\begin{prop}\label{007}
 A $2s$-nildominated supermanifold $\mathcal M$ is determined up to errors of degree $2s$ and higher by its unipotent automorphisms. In particular a nildominated supermanifold $\mathcal M$ with $H^1(M, Der^{(rk(E)-1)}_{\bar 0}(\mathcal O_{\Lambda E}))=0$ is determined by its unipotent automorphisms.
\end{prop}

\begin{proof}
 For $s=1$ the statement is trivial. Let $s> 1$ and $(i,j)$ with $U_{ij}:=U_i\cap U_j\neq \emptyset$. 
Let $(\alpha_{ij})_{ij} \in Z^1(M,G_E)$ represent $\alpha$  and assume that $(\gamma_{ij})_{ij} \in Z^1(M,G_E)$ is another cochain yielding a supermanifold as described in the proposition. Then there is a $(\beta_{ij})_{ij}\in C^1(M,G_E)$ with $\beta_{ij}\alpha_{ij}=\gamma_{ij}$ and $\beta_{ij}X_i=X_i\beta_{ij}$ for all $X \in \mathcal V_{\mathcal M,\bar 0}^{(2)}$ due to (\ref{eq:001}). Let $Y_{ij}=\log(\beta_{ij}) \in  Der^{(2)}_{\bar 0}(\mathcal O_{\Lambda E})(U_{ij})$. Now: 
\begin{align}\label{eq:002}
0=[X_i,\beta_{ij}]=\sum_{n=0}^\infty\frac{1}{n!}[X_i,Y_{ij}^n]=\sum_{n=1}^\infty\frac{1}{n!}\sum_{l=0}^{n-1} Y_{ij}^l [X_i,Y_{ij}] Y_{ij}^{n-l-1}
\end{align}
Decompose $[X_i,Y_{ij}]=\sum_{t=2}^s Z_{2t}$. Now for reasons of degree $Z_4=0$ follows directly from (\ref{eq:002}). Assuming that $Z_{2t}=0$ for all $t<u$ we find for reasons of degree the only contribution to the degree $2u$ term in (\ref{eq:002}) is in the summand for $n=1$. We have $Z_{2u}$ hence being zero. So $ [X_i,Y_{ij}]=0$ for all $X \in \mathcal V_{\mathcal M,\bar 0}^{(2)}$ and from $2s$-nildominance follows $Y_{ij}\in Der^{(2s)}_{\bar 0}(\mathcal O_{\mathcal M}) (U)$. 
\end{proof}

Finally we prove a local criterion for $2s$-nildominance for later application. 

\begin{prop}\label{ref2}
A complex supermanifold $\mathcal M$ is  $2s$-nildominated if and only if (\ref{eq:003}) holds for an arbitrary polydisc $U$ in each of the connected components of the underlying manifold $M$.
\end{prop}

\begin{proof}
Assume  $rk(E)\geq 2$. We introduce the integer valued function $F:M\to \mathbb N$ by setting $F(p)$ for $p\in M$ to be the maximal  integer $s$ such that any open neighborhood $V$ of $p$ contains an open set $U\subset V$  where (\ref{eq:003}) holds for $s$. Note that $F$ is bounded by $1$ and $\lfloor\frac{1}{2}rk(E)\rfloor$ and hence well-defined. 

\smallskip
Let $m\in \mathbb N$, $p_n\in M$ a sequence with limit $p \in M$ and $F(p_n)=m$ for all $n$. Further let $V$ be an arbitrary neighborhood of $p$. Then there is an $n_0$ such that $p_{n_0} \in V$. Further $V$ is a neighborhood of $p_{n_0}$ and hence contains an open $U\subset V$ where (\ref{eq:003}) holds for $m$. So $F(p)\geq F(p_{n_0})$. Setting $T:=\max\{F(M)\}$ we have that $F^{-1}(T)$ is closed.

\smallskip
Now let $q \in F^{-1}(T)$ and let $V_n$, $n\in \mathbb N$, be a base of neighborhoods of $q$. Let $k_n\in \mathbb N$ be the  maximal $s$ such that (\ref{eq:003}) holds in $V_n$ for $s$. Now $k_n$ is decreasing and  becomes stationary for $n\geq n_0$ with value $k$. In particular on $V_{n_0}$ there exists a $Z\in Der^{(2k)}_{\bar 0}(\mathcal O_{\Lambda E})$ with $[X,Z]=0$ for all $X \in \mathcal V_{\mathcal M,\bar 0}^{(2)}$. So on any open set $U\subset V_{n_0}$ the existence of $Z|_U$ forces (\ref{eq:003}) to be false for $k+1$. In particular $F(q)\le k$ and hence $k=T$. Again the existence of $Z$ yields that $F(V_{n_0})=k=T$. So $F^{-1}(T)$ is open.

\smallskip 
So $F$ is constant on connected components. The result follows since a polydisc is Stein. 
\end{proof}

In particular for a connected $M$ the maximal $s$ such that $\mathcal M$ is (graded) $2s$-nildominated can be determined in any local coordinate chart as soon as the restricted global vector fields are known. 

\section{Graded nildominated supermanifolds}\label{sec3}

It is natural to consider a graded version of nildominance. 

\begin{de}
   A complex supermanifold $\mathcal M$ is called graded $2s$-\textit{nildominated}, $s \in \mathbb N$, if for any open set $U\subset M$ we have:
\begin{align}\label{eq:004}
 \bigcap_{X\in \mathcal V_{\mathcal M,\bar 0}^{(2)}} 
 Ker\Big([X_\bullet,\cdot]:Der^{(2)}_{\bar 0}(\mathcal O_{\Lambda E}) (U)
 \to Der^{(2)}_{\bar 0}(\mathcal O_{\Lambda E}) (U)\Big)\subset Der^{(2s)}_{\bar 0}(\mathcal O_{\Lambda E}) (U)
\end{align}
\end{de}
Again if the odd dimension of $\mathcal M$ is $2s$ or $2s+1$ then $\mathcal M$ is simply called \textit{graded nildominated}. Assuming graded $2s$-nildominance,  let $Y \in Der^{(2)}_{\bar 0}(\mathcal O_{\Lambda E}) (U)$. Now  $[X,Y]=0$  for all $X \in \mathcal V_{\mathcal M,\bar 0}^{(2)}$ includes $[X_\bullet,Y_\bullet]=0$  for all $X \in \mathcal V_{\mathcal M,\bar 0}^{(2)}$. Hence $Y_\bullet \in Der^{(2s)}_
{\bar 0}(\mathcal O_{\Lambda E}) (U)$ and $Y \in Der^{(2s)}_{\bar 0}(\mathcal O_{\mathcal M}) (U)$. So graded $2s$-nildominance is a stronger assumption than $2s$-nildominance. In fact it is too strong: 

\begin{prop}\label{ref}
 A graded $2s$-nildominated supermanifold $\mathcal M$ is split up to error of degree $2s$ and higher. In particular a graded nildominated supermanifold $\mathcal M$ satisfying the condition $H^1(M, Der^{(rk(E)-1)}_{\bar 0}(\mathcal O_{\Lambda E}))=0$ is split.
\end{prop}

\begin{proof}
 Let $\alpha=\exp(Y)\in H^1(M,G_E)$ be the associated non-split class. We can follow from (\ref{eq:001}) that $[X_\bullet,Y_\bullet]=0$ for all $X \in \mathcal V_{\mathcal M,\bar 0}^{(2)}$. Hence $Y \in C^1(M, Der^{(2s)}_{\bar 0}(\mathcal O_{\Lambda E}))$. The cocycle condition for $\alpha$ yields the second statement.
\end{proof}

An analog of Proposition \ref{ref2} exists with similar arguments also for graded $2s$-nildominance. Further we obtain:

\begin{cor}\label{cref}
Let $\mathcal M$ be a nildominated supermanifold of odd dimension $2s$ or $2s+1$. If $\mathcal M$ is graded  $(2s-2)$-nildominated then it is split up to errors of degree $2s$.
\end{cor}

\begin{proof}
For any local derivation $Y \in (Der^{(2s-2)}_{\bar 0}(\mathcal O_{\Lambda E})\backslash Der^{(2s)}_{\bar 0}(\mathcal O_{\Lambda E})) (U)$ there is an $X \in \mathcal V_{\mathcal M,\bar 0}^{(2)}$ with $[X,Y]\neq 0$. Locally decompose $X=X_2+X_4+\cdots$ and $Y=Y_{2s-2}+Y_{2s}$ according to the $\mathbb Z$-grading in coordinates. Reasons of degree yield $[X_\bullet,Y]\neq 0$ so $\mathcal  M$ is graded $2s$-nildominated and Proposition \ref{ref} yields the statement.
\end{proof}

Proposition \ref{ref} produces the question, whether there exist $2s$-nildominated supermanifolds that is not split up to errors of degree $2s$ and higher, and in particular not graded $2s$-nildominated. The answer is positive, an example is given in the section \ref{sec4}.

\section{A $4$-nildominated example, non-split in degree $2$}\label{sec4}

Due to Corollary \ref{cref} there is no nildominated example of odd dimension smaller than $6$ that is not split up to errors of highest even degree. Here we construct a $4$-nildominated counter example in odd dimension $7$ on $M=\mathbb P^1(\mathbb C)$ that is non-split deformed already in degree $2$. We use the standard coordinate charts $U_i=\{[z_0:z_1] \ | \ z_i\neq 0\}$ for $i=0,1$ and denote $z=\frac{z_1}{z_0}$ and $w=\frac{z_0}{z_1}$. It should be mentioned that there are several conditions that an example of this kind has to satisfy reducing the hope for a technically easy example. First there have to be enough fields in $\mathcal V_{ \mathcal M,\bar 0}^{(2)}$ to ensure $4$-nildominance. Secondly there have to be not too many of them to prevent graded $4$-nildominance which would include a splitting up to errors of degree $4$ and higher. Then $H^1(M,Der_{\bar 0}^{(2)}(\mathcal O_{\Lambda E}))$ has to contain candidates for deformations that allow the elements in $\mathcal V_{ \mathcal M,\bar 0}^{(2)}
$ to produce $4$-nildominance from graded $2$-nildominance. It turns out that $E=3TM^{\otimes 2}\oplus 4TM^\ast$ yields a good starting point. The deformation has to be chosen such that for some global vector fields $X_i-X_{\bullet,i}$ is forced to be nontrivial on $U_0$ and for some to be nontrivial on $U_1$. Otherwise it can be shown that $4$-nildominance is not possible.

\bigskip
For a line bundle $\mathcal O(k)$ we use bundle coordinates $(z,\xi)$ and $(w,\xi^\prime)$ on $U_0$, resp. $U_1$,  transforming with $w=\frac{1}{z}$ and $\xi^\prime=w^k \xi$. In this notion $TM\cong\mathcal O(2)$. We use the induced identifications $H^0(M,\mathcal O(k))\cong \mathbb C[z]_{\leq k}$ and $H^1(M,\mathcal O(k))=\frac{1}{z}\mathbb C[\frac{1}{z}]_{\leq -2-k}$ with functions on $U_0$, resp. $U_0\cap U_1$. Set $E=3\mathcal O(4)\oplus4\mathcal O(-2)$. On $U_0$ denote the fiber coordinates on $3\mathcal O(4)$ by $\theta_1,\theta_2,\theta_3$, and the fiber coordinates on $4\mathcal O(-2)$ by $\eta_1,\ldots,\eta_{4}$. Following 
Proposition \ref{ref2} we will do the concrete calculations on $U_0$. By direct calculation we obtain for the global fields:

\begin{lem}\label{lele1}
The elements in $\mathcal V_{\Lambda E,\bar 0}^{(2)}$ are represented on $U_0$ by sums of the terms in the following table.  The indexes $i,\ldots,m$ cover all allowed entries skipping elements that vanish due to nilpotency of the odd variables.
\begin{align*}
\begin{array}{ll|l|l}
 \mbox{degree }2&&\mbox{degree }4 &\mbox{degree }6\\ \hline
 \mathbb C[z]_{\leq 4}\theta_i\eta_j\partial_z  & \mathbb C[z]_{\leq 10}\theta_i\theta_j\partial_z  & \mathbb C[z]_{\leq 0}\theta_i\eta_j\eta_k\eta_l \partial_z  &  \mathbb C[z]_{\leq 2}\theta_j\theta_k\eta_1 \eta_2\eta_3\eta_4 \partial_z \\ 
&&\mathbb C[z]_{\leq 6}\theta_i\theta_j\eta_k\eta_l \partial_z  &\mathbb C[z]_{\leq 8}\theta_1\theta_2\theta_3 \eta_i\eta_j\eta_k \partial_z\\ 
&&\mathbb C[z]_{\leq 12}\theta_1\theta_2\theta_3\eta_j \partial_z  &\\ \hline
  \mathbb C[z]_{\leq 2}\theta_i\theta_j\eta_k\partial_{\theta_l} & \mathbb C[z]_{\leq 8}\theta_1\theta_2\theta_3\partial_{\theta_i} & \mathbb C[z]_{\leq 4}\theta_1\theta_2\theta_3\eta_i\eta_j\partial_{\theta_k} &\mathbb C[z]_{\leq 0}\theta_1\theta_2\theta_3\eta_1\eta_2\eta_3\eta_4\partial_{\theta_i}\\ \hline
  \mathbb C[z]_{\leq 2}\theta_i\eta_j\eta_k\partial_{\eta_l}  & \mathbb C[z]_{\leq 8}\theta_i\theta_j\eta_k\partial_{\eta_l}  &   \mathbb C[z]_{\leq 4}\theta_i\theta_j\eta_k\eta_l\eta_m\partial_{\eta_n} &  \mathbb C[z]_{\leq 6}\theta_1\theta_2\theta_3\eta_1\eta_2\eta_3\eta_4\partial_{\eta_i}\\
 \mathbb C[z]_{\leq 14}\theta_1\theta_2\theta_3\partial_{\eta_i} &&\mathbb C[z]_{\leq 10}\theta_1\theta_2\theta_3\eta_i\eta_j\partial_{\eta_k} &  
 \end{array}
\end{align*}
The local common kernel in the sense of (\ref{eq:003}) as well as (\ref{eq:004}) on $U_0$ of the global vector fields contains $ \theta_1\theta_2\theta_3\partial_{\eta_i}$, $i=1,\ldots,4$. 
\end{lem}

Hence the constructed split supermanifold $(M,\mathcal O_{\Lambda E})$ is  $2$-nildominated as well as graded $2$-nildominated, the lowest regularity satisfied by any  complex supermanifold. Now we deform the split supermanifold $(M,\mathcal O_{\Lambda E})$ with 
$$Y_{01}=\Big(\frac{1}{z^2}\eta_1+\frac{1}{z^8}\eta_2\Big)\eta_3\eta_4\partial_{\theta_1} \in H^1(M,Der_{\bar0}^{(2)}(\mathcal O_{\Lambda E}))$$ 
yielding the cohomologically non-trivial: 
$$\alpha_{01}=\exp(Y_{01})=Id+\Big(\frac{1}{z^2}\eta_1+\frac{1}{z^8}\eta_2\Big)\eta_3\eta_4\partial_{\theta_1}  \ \in\ H^1(M,G_E)$$
We denote the new non-split supermanifold by $\mathcal M$. First we regard the global vector fields with degree 2 part $P_j(z) \theta_1\eta_j\partial_z$, $P_j(z)\in \mathbb C[z]_{\leq 4}$, $j=1,2$. For satisfying (\ref{eq:001}) we need to adjust the fields. Since higher degree terms vanish we obtain:
\begin{align*}
 \alpha_{01}(P_j(z)\theta_1\eta_j\partial_z)\alpha_{01}^{-1}&=P_j(z)\theta_1\eta_j\partial_z+\big[Y_{01},P_j(z)\theta_1\eta_j\partial_z\big]\\&=P_j(z)\theta_1\eta_j\partial_z+\eta_1\eta_2\eta_3\eta_4\cdot \left\{\begin{array}{rl}\Big(-\frac{P_1(z)}{z^8}\partial_z+8\frac{P_1(z)}{z^9}\theta_1\partial_{\theta_1}\Big)& \mbox{ for } j=1\\ \Big(\frac{P_2(z)}{z^2}\partial_z-2\frac{P_2(z)}{z^3}\theta_1\partial_{\theta_1}\Big)& \mbox{ for } j=2\end{array}\right.
\end{align*}
So in degree to obtain a global vector field we have to correct $P_j(z)\theta_1\eta_j\partial_z$ by adding degree 4 terms on $U_0$, resp. $U_1$. For $j=1$ we need $P_1(z)\in \mathbb C[z]_{\leq 1}$ allowing corrections on $U_1$. For $j=2$ we need $P_2(z)\in z^3\mathbb C[z]_{\leq 1}$ allowing corrections on $U_0$. 

\begin{lem}\label{lelex}
 Corrections of degree $2$ fields exist if and only if correction terms of (pure) degree 4 exist. The following degree $2$ vector fields yields global vector fields on $\mathcal M$ and need not be corrected on $U_0$ (but possibly on $U_1$):
 \begin{align*}
  \begin{array}{l|l||}
   vector field & condition\\ \hline 
  \mathbb C[z]_{\leq 4}\theta_i\eta_j\partial_z  & j \notin \{1,2\} \\
  \mathbb C[z]_{\leq 1}\theta_i\eta_1\partial_z  &  \\
  \mathbb C{[z]_{\leq 7}}\theta_i\theta_j\partial_z  &  \\
  \mathbb C[z]_{\leq 2}\theta_i\theta_j\eta_k\partial_{\theta_l}  &1\notin\{i,j\} \mbox{ or } k \notin \{1,2\} \\
  \mathbb C[z]_{\leq 0}\theta_1\theta_j\eta_1\partial_{\theta_l}  &  \\
  \mathbb C[z]_{\leq 0}\theta_1\theta_2\theta_3\partial_{\theta_i}& \\
  \mathbb C[z]_{\leq 2}\theta_i\eta_j\eta_k\partial_{\eta_l}  & \{j,k\}\neq \{1,l\}\\&\qquad \mbox{ and }\{j,k\}\neq \{2,l\}
  \end{array}
    \begin{array}{l|l}
   vector field & condition\\ \hline 
   \mathbb C[z]_{\leq 0}\theta_i\eta_1\eta_j\partial_{\eta_j}  &  \\
  \mathbb C[z]_{\leq 8}\theta_i\theta_j\eta_k\partial_{\eta_l}  & k\neq l \mbox{ and } k\in \{3,4\} \\
  \mathbb C[z]_{\leq 0}\theta_i\theta_j\eta_k\partial_{\eta_k}  & k\neq 2 \\
  \mathbb C[z]_{\leq 6}\theta_i\theta_j\eta_2\partial_{\eta_2}  & 1\notin\{i,j\} \\
  \mathbb C[z]_{\leq 0}\theta_1\theta_i\eta_2\partial_{\eta_2}  &  \\
  \mathbb C[z]_{\leq 6}\theta_i\theta_j\eta_1\partial_{\eta_k}  &  k\neq 1\\
  \mathbb C[z]_{\leq 0}\theta_i\theta_j\eta_2\partial_{\eta_k}  &  k\neq 2\\
  \mathbb C[z]_{\leq 6}\theta_1\theta_2\theta_3\partial_{\eta_i} & 
  \end{array}
 \end{align*}
  The common kernel of $[X,\cdot]: Der_{4}(\mathcal O_\mathcal M) (U_0)
 \to Der_{6}(\mathcal O_\mathcal M) (U_0)$ with $X$ running through the list is the $\mathcal O_M(U_0)$-module spanned by  $\eta_i\theta_1\theta_2\theta_3\partial_z$, $\theta_2\theta_3\eta_1\eta_3\eta_4\partial_{\theta_1}$, and  $\eta_i\eta_j\theta_1\theta_2\theta_3\partial_{\eta_k}$. 
\end{lem}

\begin{proof}
 The first statement follows from $Y_{01}DY_{01}=Y_{01}Y_{01}=0$ for any derivation $D$. The list only contains fields that either commute with $Y_{01}$ or yield degree $4$ terms that can be corrected by adding a degree $4$ term on $U_1$. One inclusion of the last statement follows by direct calculation. For the converse inclusion assume that $$Z=\sum_{IJ} f_{IJ}\theta^I\eta^J\partial_{z}+\sum_{IJk} g_{IJk}\theta^I\eta^J\partial_{\theta_k}+\sum_{IJk} h_{IJk}\theta^I\eta^J\partial_{\eta_k}$$ with numerical coefficient functions is in the common kernel. Now $[\theta_1\theta_2\theta_3\partial_{\eta_1},Z]$ contains  $f_{(000)(1111)}\theta^{(111)}\eta^{(0111)}\partial_z$. Hence $f_{(000)(1111)}=0$. Further $[\theta_1\theta_2\eta_1\partial_{\eta_2},Z]$ contains the summand $f_{(001)(0111)}\theta^{(111)}\eta^{(1011)}\partial_z$. So varying the indexes in $\theta_i\theta_j\eta_k\partial_{\eta_l}$, by analog arguments $f_{IJ}=0$ for $|I|=1$. Further $[\theta_1\eta_2\eta_3\partial_{\eta_4},Z]$ contains $f_{(011)(1001)} 
 \theta^{(111)}\eta^{(1110)}\partial_z $. So analogously varying the indexes of $\theta_i\eta_j\eta_k\partial_{\eta_l}$ in the allowed range, we obtain  $f_{IJ}=0$ for $|I|=2$. In particular the first summand  of $Z$ is a $\mathcal O_M(U_0)$-linear combination of $\eta_i\theta_1\theta_2\theta_3\partial_z$. Now $[\theta_1\theta_2\eta_1\partial_{\eta_1},Z]$ contains $\sum_k g_{(001)(1111)k}\theta^{(111)}\eta^{(1111)}\partial_{\theta_k}$. Altogether $g_{IJk}=0$ for $|I|=1$ by varying $i,j$ in   $\theta_i\theta_j\eta_1\partial_{\eta_1}$. Further $[\theta_1\eta_1\eta_2\partial_{\eta_2},Z]$ contains $\sum_k g_{(011)(0111)k}\theta^{(111)}\eta^{(1111)}\partial_{\theta_k}$. Since $\theta_i\eta_2\eta_j\partial_{\eta_j}$ is not allowed we obtain  $g_{IJk}=0$ for $|I|=2$ and $J\neq (1011)$.
 Now $[\theta_2\theta_3\eta_2\partial_{\theta_2},Z]$ contains $\ g_{(110)(1011)1 }\theta^{(111)}\eta^{(1111)}\partial_{\theta_1}$ and $\ g_{(110)(1011)3 }\theta^{(111)}\eta^{(1111)}\partial_{\theta_3}$. With $\theta_2\theta_3\eta_2\partial_{\theta_3}$ it follows $g_{IJk}=0$ for $|I|=2$, $J=(1011)$ and $(I,k)\in \{((110),1),((110),3),((101),1),((101),2)\}$.
  Now $[\theta_1\theta_2\partial_z,Z]$ contains $(g_{(011)(1011)2}+g_{(101)(1011)1})\theta^{(111)}\eta^{(1011)}\partial_{z}$. In particular $g_{(011)(1011)2}=0$. Further $[\theta_1\theta_3\partial_z,Z]$ contains $(g_{(110)(1011)1}-g_{(011)(1011)3})\theta^{(111)}\eta^{(1011)}\partial_{z}$ hence $g_{(011)(1011)3}=0$.   
 We switch to the third summand of the derivation $Z$.  
 The element $[\theta_1\theta_2\theta_3\partial_{\theta_1},Z]$ contains $\sum_kh_{(100)(1111)k}\theta^{(111)}\eta^{(1111)}\partial_{\eta_k}$ so varying $\theta_1\theta_2\theta_3\partial_{\theta_i}$ observe $h_{IJk}=0$ for $|I|=1$. Now we can continue with the analysis of the second part of $Z$. The element $[\theta_2\theta_3\eta_2\partial_{\theta_2},Z]$ contains $(g_{(110)(1011)2}-g_{(110)(1011)2}-g_{(101)(1011)3})\theta^{(111)}\eta^{(1111)}\partial_{\theta_2}$. So $g_{(101)(1011)3}=0$. Now   $[\theta_2\theta_3\partial_z,Z]$ contains $(g_{(110)(1011)2}+g_{(101)(1011)3})\theta^{(111)}\eta^{(1011)}\partial_{z}$. So $g_{(110)(1011)2}=0$ and  $g_{(011)(1011)1}$ is the only remaining $g_{IJk}$ with $|I|=2$.  
  Now $[\theta_1\theta_2\eta_3\partial_{\theta_1},Z]$ contains $\sum_kh_{(101)(1101)k}\theta^{(111)}\eta^{(1111)}\partial_{\eta_k}$. Since $1\notin \{i,j\}$ in $\theta_i\theta_j\eta_2\partial_{\theta_k}$ we can only conclude $h_{IJk}=0$ for $|I|=2$ and $(I,J)\neq((011),(1011))$. Further $[\theta_1\eta_1\eta_3\partial_{\eta_2},Z]$ contains $g_{(111),(0101),1}\theta^{(111)}\eta^{(1111)}\partial_{\eta_2}$. Varying $\theta_i\eta_j\eta_k\partial_{\eta_l}$ with $k\neq l$ we obtain $g_{IJk}=0$ for $|I|=3$. Now $[\theta_1\eta_1\eta_2\partial_{\eta_3},Z]$ contains the term $h_{(011)(1011)1}\theta^{(111)}\eta^{(1111)}\partial_{\eta_3}$. Varying $\theta_1\eta_j\eta_2\partial_{\eta_k}$ with $j\neq k$ we finally obtain $h_{IJk}=0$ for $|I|=2$.
\end{proof}

\begin{lem}\label{lele3}
The following degree $4$ vector fields yields global vector fields on $\mathcal M$ and need not be corrected on $U_0$ (but possibly on $U_1$):
 \begin{align*}
  \begin{array}{l|l||}
   vector field & condition\\ \hline 
   \mathbb C[z]_{\leq 0}\theta_i\eta_j\eta_k\eta_l \partial_z & \\
   \mathbb C[z]_{\leq 6}\theta_i\theta_j\eta_k\eta_l \partial_z & \\
   \mathbb C[z]_{\leq 12}\theta_1\theta_2\theta_3\eta_j \partial_z & j \in \{3,4\} \\
   \mathbb C[z]_{\leq 9}\theta_1\theta_2\theta_3\eta_1 \partial_z &  \\
   \mathbb C[z]_{\leq 3}\theta_1\theta_2\theta_3\eta_2 \partial_z &  
  \end{array}
    \begin{array}{l|l}
   vector field & condition\\ \hline 
   \mathbb C[z]_{\leq 4}\theta_1\theta_2\theta_3\eta_i\eta_j\partial_{\theta_k} & \\
   \mathbb C[z]_{\leq 4}\theta_i\theta_j\eta_k\eta_l\eta_m\partial_{\eta_n} & \\
   \mathbb C[z]_{\leq 10}\theta_1\theta_2\theta_3\eta_i\eta_j\partial_{\eta_k}  &\{i,j\}\neq \{1,k\} \mbox{ and }\{i,j\}\neq \{2,k\} \\
   \mathbb C[z]_{\leq 8}\theta_1\theta_2\theta_3\eta_1\eta_i\partial_{\eta_i}  & \\
   \mathbb C[z]_{\leq 2}\theta_1\theta_2\theta_3\eta_2\eta_i\partial_{\eta_i}  &
  \end{array}
 \end{align*}
The common kernel of $[X,\cdot]: Der_{2}(\mathcal O_\mathcal M) (U_0)
 \to Der^{(4)}(\mathcal O_\mathcal M) (U_0)$ with $X$ running through this list and through the list of Lemma \ref{lelex} is the $\mathcal O_M(U_0)$-module spanned by  $\theta_1\theta_2\theta_3\partial_{\eta_i}$.
\end{lem}

\begin{proof}
 The list only contains fields that either commute with $Y_{01}$ or yield degree $6$ terms that can be corrected by adding a degree $6$ term on $U_1$. One inclusion of the last statement is direct calculation. The converse inclusion can be seen as follows. Let 
 $$Z=\sum_{IJ} f_{IJ}\theta^I\eta^J\partial_{z}+\sum_{IJk} g_{IJk}\theta^I\eta^J\partial_{\theta_k}+\sum_{IJk} h_{IJk}\theta^I\eta^J\partial_{\eta_k}$$
 with numerical coefficient functions be an element of the common kernel. Now the element  $[\theta_1\theta_2\theta_3\eta_1\eta_2\partial_{\eta_3},Z]$ contains $f_{(000)(0011)}\theta^{(111)}\eta^{(1101)}\partial_z$. Varying $\theta_1\theta_2\theta_3\eta_i\eta_j\partial_{\eta_k}$ with $k\notin\{i,j\}$ yields $f_{IJ}=0$ for $|I|=0$. And $[\theta_1\theta_2\eta_1\eta_2\eta_3\partial_{\eta_3},Z]$ contains $f_{(001)(0010)}\theta^{(111)}\eta^{(1110)}\partial_z$ so varying $\theta_i\theta_j\eta_k\eta_l\eta_m\partial_{\eta_n}$ we obtain $f_{IJ}=0$ for $|I|=1$. Further $[\theta_1\theta_2\eta_3\partial_{\theta_2},Z]$ contains $f_{(011)(0000)}\theta^{(111)}\eta^{(0010)}\partial_z$. Variations of $\theta_i\theta_j\eta_3\partial_{\theta_j}$ yield $f_{IJ}=0$ for $|I|=2$. Further $[\theta_1\theta_2\theta_3\eta_1\eta_2\partial_{\eta_2}, Z]$ contains $\sum_kg_{(000)(0111)k}\theta^{(111)}\eta^{(1111)}\partial_{\theta_k}$. Varying $\theta_1\theta_2\theta_3\eta_i\eta_j\partial_{\eta_j}$ we have $g_{IJk}=0$ for $|I|=0$. Now 
 $[\theta_1\theta_2\eta_1\eta_2\eta_3\partial_{\eta_3},Z]$ contains $\sum_kg_{(001)(0011)k}\theta^{(111)}\eta^{(1111)}\partial_{\theta_k}$ and varying $\theta_i\theta_j\eta_k\eta_l\eta_m\partial_{\eta_m}$ we obtain $g_{IJk}=0$ for $|I|=1$.  And $[\theta_1\eta_1\eta_2 \partial_{\eta_3},Z]$ contains $\sum_k g_{(011)(0010)k} \theta^{(111)}\eta^{(1100)}\partial_{\theta_k}$. Variations of $\theta_i\eta_j\eta_k \partial_{\eta_l}$, $l\notin\{j,k\}$ yield $g_{IJk}=0$ for $|I|=2$.  Now $[\theta_1\theta_2\theta_3\eta_1\partial_{z},Z]$ contains $\sum_{|J|=3} h_{(000)J1}\theta^{(111)}\eta^{J}\partial_{z}$. Varying $\theta_1\theta_2\theta_3\eta_i\partial_{z}$ we have $h_{IJk}=0$ for $|I|=0$. Further $[\theta_1\theta_2\theta_3\eta_1\eta_2\partial_{\theta_1},Z]$ contains $\sum_kh_{(100)(0011)k} \theta^{(111)}\eta^{(1111)}\partial_{\eta_k}$. Varying 
 $\theta_1\theta_2\theta_3\eta_i\eta_j\partial_{\theta_k}$ we have $h_{IJk}=0$ for $|I|=1$. Now $[\theta_1\theta_2\eta_3\partial_{\theta_1},Z]$ contains $\sum_l(h_{(101)(1000)l}\eta_1+h_{(101)(0100)l}\eta_2+h_{(101)(0001)l}\eta_4)\eta_3\theta^{(111)}\partial_{\eta_l}$. Varying $\theta_i\theta_j\eta_k\partial_{\theta_j}$, $k\in \{3,4\}$ we find $h_{IJk}=0$ for $|I|=2$.  
 Further $[\theta_1\eta_1\eta_2\eta_3\partial_z, Z]$ contains $g_{(111)(0000)1}\theta^{(111)}\eta^{(1110)}\partial_{z}$. With $\theta_i\eta_4\partial_z$ we obtain $g_{IJk}=0$ for $|I|=3$.
\end{proof}

Assume that some local vector field $V \in Der_{\bar 0}^{(2)}(\mathcal O_{\Lambda E})\backslash Der_{\bar 0}^{(4)}(\mathcal O_{\Lambda E})$ is in the common kernel of the $[X,\cdot]: Der_{ \bar 0}^{(2)}(\mathcal O_\mathcal M) (U_0)
 \to Der_{ \bar 0}^{(4)}(\mathcal O_\mathcal M) (U_0)$ , $X \in \mathcal V_{\mathcal M, \bar 0}^{(2)}$. Let $V=V_{2}+V_{4}+V_6$ according to the $\mathbb Z$-grading.  For global fields $X\in \mathcal V_{\mathcal M,\bar 0}^{(2)}$ of pure degree on $U_0$ as they are displayed in Lemmas \ref{lelex} and \ref{lele3},  we have in particular the condition $[X,V_{2k}]=0$ for $k=1,2$. So by Lemma \ref{lele3}, the only  interesting case is  $V_{2}=\sum_{i=1}^4f_i\theta_1\theta_2\theta_3\partial_{\eta_i}$ with $f_i\in \mathcal O_M(U_0)$. So on the one hand, $V_4$ is forced to lie in the kernel described  in Lemma \ref{lelex}. 
 
\bigskip
Recall that the correction terms for global fields of degree 2 only appear in degree 4, see Lemma \ref{lelex}. Hence  we can assume  for all of these $X=X_2+X_4\in \mathcal V_{\mathcal M,\bar 0}^{(2)}$  that $[X_2,V_{2}]=0$ and $[X_2,V_{4}]=-[X_4,V_{2}]$. 
Regarding the following corrected global vector fields on $U_0$:
\begin{align*}
 X^{\prime,j}:=z^2\theta_1\eta_2\eta_j\partial_{\eta_j}+\theta_1\eta_1\eta_2\eta_3\eta_4\partial_{\theta_1}, \ j\in\{3,4\}
\end{align*}
we have: 
\begin{align*}
 -[X_4^{\prime,j},V_{2} ]=-\sum_{i=1}^4f_i\theta_1\theta_2\theta_3\eta_1\eta_2\eta_3\eta_4\partial_{\eta_i} =[z^2 \theta_1\eta_2\eta_j\partial_{\eta_j},V_{4}]
\end{align*}
So on the other hand, for $i\neq j$ the summand ${f_i}\theta_1\theta_2\theta_3\eta_1\eta_2\eta_3\eta_4\partial_{\eta_i}$ can only be created by $\frac{f_i}{z^2}\theta_2\theta_3\eta_1\eta_3\eta_4\partial_{\eta_i}$ as a summand of $V_4$. Now Lemma \ref{lelex} yields $f_i=0$ for all $i$ and hence $V_2=0$ contradicting the above assumption. We can directly follow:
\begin{prop}
 The complex supermanifold $\mathcal M$ is graded $2$-nildominated but not graded $4$-nildominated. At the same time $\mathcal M$ is $4$-nildominated. Further it is non-split being deformed already in degree 2.
\end{prop}

\section{Strictly nildominated supermanifolds}\label{sec5}

Nildominance is not stable under products in the category of complex supermanifolds. We modify our approach.
\begin{de}
 A complex supermanifold $\mathcal M$ is called strictly $t$-\textit{nildominated}, $t \in \mathbb N$, if for any open set $U\subset M$ we have:
\begin{align}
 \bigcap_{X\in \mathcal V_{\mathcal M,\bar 0}^{(2)}} 
  Ker([X,\cdot]:Der(\mathcal O_\mathcal M) (U) 
 \to Der(\mathcal O_\mathcal M) (U))=Der^{(t)}(\mathcal O_\mathcal M) (U)\label{eq:1}
\end{align}
\end{de}

If its odd dimension is $t+1$ then it is called strictly \textit{nildominated}. In particular any strictly  $(2s-1)$- or strictly $2s$-nildominated supermanifold is $2s$-nildominated and following Proposition \ref{007} it is determined up to errors of degree $2s$ and higher by its unipotent automorphisms.

\begin{prop}\label{pro1}
Let $\mathcal M^\prime$ and $\mathcal M^{\prime\prime}$  be $t^\prime$-, resp. $t^{\prime\prime}$-nildominated supermanifolds associated with the bundles $E^\prime\to M^{\prime}$, resp. $E^{\prime\prime}\to M^{\prime\prime}$.   Let $\alpha^\prime\in H^1(M^\prime,G_{E^\prime})$ and $\alpha^{\prime\prime}\in H^1(M^{\prime\prime},G_{E^{\prime\prime}})$ be the associated classes. Set $E:=E^\prime\oplus E^{\prime\prime}\to M^{\prime}\times M^{\prime\prime}=:M$ and $t=\min\{t^\prime,t^{\prime\prime}\}$. There is a supermanifold $\mathcal M$ on $E\to M$ with morphisms $\mathcal M\to \mathcal M^\prime$ and $\mathcal M\to \mathcal M^{\prime\prime}$ reducing to the projections on the product of vector bundles. Further $\mathcal M$ is strictly $t$-nildominated and well-defined up to errors in $C^1(M,End^{(t)}_{\bar 0}(\mathcal O_{\Lambda E}))$ by these properties.
\end{prop}

\begin{proof}
 Existence is given by $\exp(\log(\alpha^\prime)+\log(\alpha^{\prime\prime}))$. Let $(z_i^{\prime},\xi_\sigma^{\prime})$, $(z_j^{\prime\prime},\xi_\rho^{\prime\prime})$ be local coordinates for $E^\prime\to M^\prime$, resp. $E^{\prime\prime}\to M^{\prime\prime}$.  Further let
 $Y=\sum_{I}\xi^{\prime\prime I} Y^\prime_I+\sum_{J}\xi^{\prime J} Y^{\prime\prime}_J $
 be a  general  local derivation in $Der_{\bar 0}^{(2)}(\mathcal O_{\Lambda E})$ with $Y^\prime_I \in Der(\mathcal O_{\Lambda E^\prime})$ and $Y^{\prime\prime}_J \in Der(\mathcal O_{\Lambda E^{\prime\prime}})$. For $X^\prime \in \mathcal V_{\mathcal M^\prime,\bar 0}^{(2)}$ we have $[X^\prime,Y]=\sum_{I}\xi^{\prime\prime I} [X^\prime,Y^\prime_I]+\sum_{J}X^\prime(\xi^{\prime J}) Y^{\prime\prime}_J$. Assuming $[X^\prime,Y]=0$ for all $X^\prime \in \mathcal V_{\mathcal M^\prime,\bar 0}^{(2)}$ we find by strict $t^\prime$-nildominance of $\mathcal M^\prime$ that  $Y^\prime_I \in Der^{(t^\prime)}(\mathcal O_{\Lambda E^\prime})$. Analogously assuming $[X^{\prime\prime},Y]=0$ for all $X^{\prime\prime} \in \mathcal V_{\mathcal M^{\prime\prime},\bar 0}^{(2)}$ yields $Y^{\prime\prime}_J \in Der^{(t^{\prime\prime})}(\mathcal O_{\Lambda E^{\prime\prime}})$. Hence:
\begin{align}\label{eq:10}
 Y \in Der^{(t^\prime)}(\mathcal O_{\Lambda E^\prime})\oplus Der^{(t^{\prime\prime})}(\mathcal O_{\Lambda E^{\prime\prime}})\oplus End^{(t)}(\mathcal O_{\Lambda E})
\end{align}
Setting $Y=Y_{ij}$ we proceed as in the proof of Proposition \ref{007} and obtain that $\mathcal M$ is well-defined up to terms in $C^1(M,End^{(t)}_{\bar 0}(\mathcal O_{\Lambda E}))$. Further (\ref{eq:1}) follows from (\ref{eq:10}). 
\end{proof}

\begin{cor}
The product of two strictly $t$-nildominated supermanifolds is strictly $t$-nildominated.
\end{cor}

An analog of Proposition \ref{ref2} exists with similar arguments also for strict $2s$-nildominance. 
Using this we return to the above example:

\begin{prop}
The complex supermanifolds $\mathcal M$ in section \ref{sec4} is strictly $3$-nildominated and non-split being deformed already in degree 2.
\end{prop}

\begin{proof}
Regard the common kernel of $[X,\cdot]: Der_q(\mathcal O_\mathcal M)(U_0)\to Der(\mathcal O_\mathcal M)(U_0)$ for $X$ in the lists of Lemmas \ref{lelex} and \ref{lele3} and further of the global degree $6$ vector fields in Lemma \ref{lele1}. For $q=-1$ let $Z=\sum_kg_k\partial_{\theta_k}+\sum_{k}h_k\partial_{\eta_k}$ be in the common kernel and regard the commutators with the fields $\theta_1\theta_2\theta_3\partial_{\theta_i}$ to see that the $g_k$ vanish. The commutators with $\theta_1\theta_2\eta_i\partial_{\eta_3}$ then show that the $h_k$ vanish. For $q=0$ let $Z=f\partial_z+\sum_{ij}g_{ij}\theta_i\partial_{\theta_j}+\sum_{ij}\hat g_{ij}\eta_i\partial_{\theta_j}+\sum_{ij}\hat h_{ij}\theta_i\partial_{\eta_j}+\sum_{ij}h_{ij}\eta_i\partial_{\eta_j}$ be in the common kernel. Then  $[\theta_1\theta_2\theta_3\eta_1\eta_2\eta_3\eta_4 \partial_{\theta_i},Z]$ and  $[\theta_1\theta_2\theta_3\eta_1\eta_2\eta_3\eta_4 \partial_{\eta_i},Z]$ yield $\hat h_{ij}=\hat g_{ij}=0$ for all $i,j$ and $g_{ij}=h_{ij}=0$ for $i\
neq j$. Further the $\partial_{\theta_k}$, resp. $\partial_{\eta_k}$, terms in $[\theta_i\theta_j\eta_1\eta_2\eta_3\eta_4 \partial_z,Z]$, resp. $[\theta_1\theta_2\theta_3\eta_i\eta_j\eta_k \partial_z,Z]$, yield $g_{ii},h_{ii} \in \mathbb C$. 
Further the coefficients of $\partial_z$ in   $[\theta_i\theta_j\partial_z,Z]=0$ yield $f^\prime=g_{ii}+g_{jj}$ for all $i\neq j$. In particular $g_{ii}=g_{jj}$ for all $i,j$ and $f^\prime=2g_{ii}$.  Now  $0=[z\theta_i\theta_j\partial_z,Z]=z[\theta_i\theta_j\partial_z,Z]-Z(z)\theta_i\theta_j\partial_z=Z(z)\theta_i\theta_j\partial_z$ and $Z(z)=f$  yield in addition $f=0$ and so $g_{ii}=0$ for all $i$. From $[\theta_2\theta_3\eta_i\partial_{\theta_1},Z]$ we obtain $h_{ii}=0$ for all $i$. For $q=1$ let $$Z=\sum_{IJ} f_{IJ}\theta^I\eta^J\partial_{z}+\sum_{IJk} g_{IJk}\theta^I\eta^J\partial_{\theta_k}+\sum_{IJk} h_{IJk}\theta^I\eta^J\partial_{\eta_k}$$ be in the common kernel. The $\partial_z$ terms  in $[\theta_1\theta_2\theta_3\partial_{\theta_i}, Z]$ and $[\theta_2\theta_3\eta_i\partial_{\eta_i},Z]$ yield $f_{IJ}=0$ for all indexes. The $\partial_{\eta_k}$ terms in $[\theta_1\theta_2\theta_3\partial_{\theta_i}, Z]$ show $h_{IJk}=0$ for $|I|=|J|=1$. The $\partial_{\theta_k}$ terms in $[\theta_i\theta_j\eta_l\
partial_{\eta_l},Z]$ yield $g_{IJk}=0$ for $|I|=|J|=1$ and $|I|=0$. The $\partial_{\eta_k}$ terms with two $\theta$s in 
$[\theta_i\eta_j\eta_k\partial_{\eta_l},Z]$ with $l\notin \{j,k\}$ yield $g_{IJk}=0$ for $|I|=2$. The  $\partial_z$ terms in $[\theta_i\eta_j\partial_z,Z]$ yield $h_{IJk}=0$ for $k\neq 2$ and $|I|=0$ or $2$. The $\partial_z$ term in $[\theta_i\eta_2\eta_3\eta_4\partial_z,Z]$ yields $h_{IJ2}=0$ for $|I|=2$ and the $\partial_z$ term in $[\theta_1\theta_2\theta_3\eta_2\partial_z,Z]$ yields $h_{IJ2}=0$ for $|I|=0$. For $q=2$ the result follows from $4$-nildominance of $\mathcal M$.
\end{proof}


\begin{thebibliography}{18}

\bibitem[BK15]{BK} Bergner, H., Kalus, M., \emph{Automorphism groups of compact complex supermanifolds}, preprint, arXiv:1506.01295

\bibitem[DM99]{D-M} Deligne, P., Morgan, J. W., \emph{Notes on supersymmetry (following Joseph Bernstein)}, Quantum fields and strings: a course for mathematicians, Vol. 1, 2 (Princeton, NJ, 1996/1997), pp. 41-97, Amer. Math. Soc., Providence, RI, 1999

\bibitem[Gr82]{Gr} Green, P., \emph{On holomorphic graded manifolds},  Proc. Amer. Math. Soc.  85  (1982), no. 4, pp. 587-590

\bibitem[Ka16]{Kal} Kalus, M., \emph{Complex supermanifolds of odd dimension beyond 5}, arXiv:1601.07434

\bibitem[Ko77]{Kost} Kostant, B., \emph{Graded manifolds, graded Lie theory, and prequantization}, Differential geometrical methods in mathematical physics, pp. 177-306, Lecture Notes in Math., Vol. 570, Springer, Berlin, 1977

\bibitem[Le80]{Lei} Leites, D. A., \emph{Introduction to the theory of supermanifolds} (Russian),
Uspekhi Mat. Nauk 35 (1980), no. 1 (211), pp. 3-57

\bibitem[On96]{On96} Onishchik, A. L., \emph{On complex homogeneous supermanifolds} (Russian) Mathematics at Yaroslavl University (Russian), pp. 133-153, Yarosl. Gos. Univ., Yaroslavl, 1996

\bibitem[On07]{On07} Onishchik, A. L., \emph{Homogeneous supermanifolds over Grassmannians} (English summary)
J. Algebra 313 (2007), no. 1, pp. 320-342. 

\bibitem[Ro85]{Rot} Rothstein, M. J., \emph{Deformations of complex supermanifolds},  Proc. Amer. Math. Soc.  95  (1985),  no. 2, pp. 255-260

\bibitem[Vi11]{Vi11} Vishnyakova, E. G., \emph{On complex Lie supergroups and split homogeneous supermanifolds}, (English summary) Transform. Groups 16 (2011), no. 1, pp. 26-285

\bibitem[Vi15]{Vi15} Vishnyakova, E. G., \emph{The splitting problem for complex homogeneous supermanifolds},  (English summary) J. Lie Theory 25 (2015), no. 2, pp. 459-476

\end{thebibliography}
\end{document}